\documentclass{amsart}
\usepackage{amssymb,latexsym}
\usepackage{amscd,amsthm}
\newtheorem{theorem}{Theorem}[section]
\newtheorem{lemma}[theorem]{Lemma}
\newtheorem{proposition}[theorem]{Proposition}
\newtheorem{corollary}[theorem]{Corollary}

\theoremstyle{definition}
\newtheorem{definition}[theorem]{Definition}
\newtheorem{fact}[theorem]{Fact}
\newtheorem*{remark}{Remark}

\newtheorem*{example}{Example}

\DeclareMathOperator{\Ext}{Ext}
\DeclareMathOperator{\Hom}{Hom}

\DeclareMathOperator{\cok}{cok}

\newcommand{\cat}[1]{\mathcal{#1}}           

\newcommand{\class}[1]{\mathcal{#1}}   

\newcommand{\mathcolon}{\colon\,} 
\newcommand{\ar}{\xrightarrow{}} 

\begin{document}

\title{Model structures on exact categories}

\date{\today}

\author{James Gillespie}
\address{505 Ramapo Valley Road \\
         Ramapo College of New Jersey \\
         Mahwah, NJ 07430}
\email[Jim Gillespie]{jgillesp@ramapo.edu}


\begin{abstract}
We define model structures on exact categories which we call exact model structures. We look at the relationship between these model structures and cotorsion pairs on the exact category. In particular, when the underlying category is weakly idempotent complete we get Hovey's one-to-one correspondence between model structures and complete cotorsion pairs. We classify the right and left homotopy relation in terms of the cotorsion pairs and look at examples of exact model structures. In particular, we see that given any hereditary abelian model category, the full subcategories of cofibrant, fibrant and cofibrant-fibrant subobjects \emph{each} have natural exact model structures equivalent to the original model structure. These model structures each have interesting characteristics. For example, the cofibrant-fibrant subobjects form a Frobenius category whose stable category is the same thing as the homotopy category of its model structure.
\end{abstract}

\maketitle


\section{Introduction and Preliminaries}

Exact categories were introduced by Quillen in~\cite{quillen-algebraic K-theory}. These are additive categories which may not have all kernels and cokernels but which have enough structure to allow for a notion of ``short exact sequences''. The axioms allow for a rather thorough treatment of homological algebra analogous to the traditional theory in an abelian category. For example, see~\cite{buhler-exact categories}. But homological algebra itself is encompassed in Quillen's notion of a model category and so there ought to be model structures on exact categories describing homological algebra in these categories.

So in this paper we define and make a brief study of exact model structures. These are exact categories with a model structure that is compatible in a nice way with the short exact sequences. The precise statement is Definition~\ref{def-exact model structures} and is entirely analogous to Hovey's definition of an abelian model category which appeared in~\cite{hovey}. However, a model category is usually assumed to have, at least, all finite limits and colimits and more often now assumed to have all small limits and colimits. As we explain in more detail in the beginning of Section~\ref{sec-the homotopy category of an exact model structure}, we don't need the full finite limit and colimit assumptions to obtain the standard introductory results of homotopy theory, including the fundamental result on localization with respect to the class of weak equivalences. In fact all the limits and colimits needed already come in the definition of an exact category. Therefore we use the term \emph{exact model structure} and reserve \emph{exact model category} for when the category does come equipped with all small limits and colimits.

We can also define cotorsion pairs in exact categories and we see that Hovey's correspondence between abelian model structures and cotorsion pairs naturally carries over to a correspondence between exact model structures and cotorsion pairs. However, only one direction of the correspondence seems to hold for a general exact model structure. To have a perfect one-to-one correspondence between exact model structures and cotorsion pairs we need to assume the exact category is ``weakly idempotent complete''. This means the exact category has cokernels of all split monos and kernels of all split epis. More details are given in Section~\ref{subsec-weakly idempotent complete exact categories}.

This correspondence between model structures and cotorsion pairs provides a way to translate standard language from the theory of model categories to purely algebraic ideas and it is interesting to attempt to characterize certain model category notions in terms of cotorsion pairs. We find a nice algebraic characterization of left and right homotopy which appears as Proposition~\ref{prop-left and right homotopic maps in exact model structures} and Corollary~\ref{cor-characterizations of homotopic maps in inj. proj. Frob. model structures}.

We show that for an hereditary abelian model category (or more generally an hereditary exact model structure), the full subcategories of cofibrant, fibrant, and cofibrant-fibrant subobjects come equipped with their own exact model structures describing their usual homotopy categories. We can think of these as canonical sub-model structures and each one has its own interesting characteristics. The cofibrant objects have an equivalent ``injective'' sub-model structure, the fibrant objects have an equivalent ``projective'' sub-model structure and the cofibrant-fibrant subobjects have an equivalent ``Frobenius'' sub-model structure structure.

We assume the reader is interested in model categories or in the interactions between homotopy theory and algebra, but otherwise the paper is believed to be self-contained. Definitions or proper references will be given as they are needed. Section~\ref{sec-exact categories} concerns definitions and basic results on exact categories, cotorsion pairs in exact categories and weakly idempotent complete exact categories. Section~\ref{sec-exact model structures} concerns the correspondence between exact model structures and cotorsion pairs in weakly idempotent complete exact categories. In Section~\ref{sec-the homotopy category of an exact model structure} we characterize the left and right homotopy relation in terms of the cotorsion pairs. We also define and look at projective, injective and Frobenius model structures. Finally, Section~\ref{sec-examples and applications} mainly concerns the examples of the canonical sub-model structures of an hereditary exact model structure, but we also look at how classical homotopy theory of chain complexes fits into our setup.

The author would like to thank Sergio Estrada for asking questions which led to this paper and for pointing out the notion of an exact category. Thanks to T.~B\"uhler for his nicely written monograph~\cite{buhler-exact categories} which I found to be a great read on exact categories.

\section{Exact categories}\label{sec-exact categories}

An exact category is a pair $(\class{A},\class{E})$ where $\class{A}$ is an additive category and $\class{E}$ is a class of ``short exact sequences'': That is, triples of objects connected by arrows $A \xrightarrow{i} B \xrightarrow{p} C$ such that $i$ is the kernel of $p$ and $p$ is the cokernel of $i$. A map such as $i$ is necessarily a monomorphism and in the language of exact categories is called an \emph{admissible monomorphism} while $p$ is called an \emph{admissible epimorphism}. The class $\class{E}$ of short exact sequences must satisfy the following axioms which are inspired by the properties of short exact sequences in any abelian category:

\begin{enumerate}

\item $\class{E}$ is closed under isomorphisms.

\item $\class{E}$ contains each of the canonical split exact sequences $A \rightarrow A \oplus B \rightarrow B$.

\item Any pushout of an admissible monomorphism exists and admissible monos are stable under pushouts. Similarly, any pullback of an admissible epimorphism exists and admissible epis are stable under pullbacks.

\item Admissible monomorphisms are closed under compositions. Similarly, admissible epimorphisms are closed under compositions.

\end{enumerate}

We sometimes denote admissible monomorphisms by $\rightarrowtail$ and denote admissible epimorphisms by $\twoheadrightarrow$.
These axioms are equivalent to Quillen's original definition in~\cite{quillen-algebraic K-theory}. See B\"uhler's recent paper~\cite{buhler-exact categories} for a very thorough and readable exposition on exact categories.

Given any additive category $\class{A}$, we may take $\class{E}$ to be all split exact sequences to get a (trivial) exact category $(\class{A},\class{E})$. However, taking $\class{E}$ to be \emph{all} short exact sequences that already exist in $\class{A}$ will not, in general, define an exact category. Certainly it does when $\class{A}$ is abelian. The typical exact category arises as a full subcategory of an abelian category: Given any strictly full subcategory $\class{A}$ of an abelian category $\class{B}$, in which $\class{A}$ is closed under extensions, gives an exact category $(\class{A},\class{E})$ where $\class{E}$ consists of the short exact sequences from $\class{B}$ in which all three terms are objects in $\class{A}$. Conversely, any exact category may be embedded inside an abelian category. See~\cite{buhler-exact categories} for more details on all of the above.

\subsection{Cotorsion pairs in exact categories} Let $\class{A} = (\class{A},\class{E})$ be an exact category. In analogy to abelian categories, the axioms allow for the usual construction of the Yoneda Ext bifunctor $\Ext^1_{\class{A}}(M,N)$. It is the abelian group of equivalence classes of short exact sequences $N \rightarrowtail Z \twoheadrightarrow M$. See Chapter~XII.4 of MacLane~\cite{homology} for details on the construction of the Yoneda Ext bifunctor. In particular, we get that $\Ext^1_{\class{A}}(M,N) = 0$ if and only if every short exact sequence $N \rightarrowtail Z \twoheadrightarrow M$ is isomorphic to the split exact sequence $N \rightarrowtail N \oplus M \twoheadrightarrow M$.

We say an object $I \in \class{A}$ is injective if any admissible monomorphism $I\rightarrowtail Z$ splits (has a left inverse). Since the arrow $I \rightarrowtail Z$ has a cokernel,  $I$ must be a direct summand of $Z$ by (the dual of) the argument in Remark~7.4 of~\cite{buhler-exact categories}. From this one can see $I$ is injective if and only if $\Ext^1_{\class{A}}(M,I) = 0$ for any $M \in \class{A}$. Projective objects are defined dually. (Perhaps these objects should be called ``admissible injectives'' and ``admissible projectives'', since they may not be categorically injective or projective. But it seems as though the language is standard for exact categories). The definition of a cotorsion pair also readily generalizes to exact categories.

\begin{definition}\label{def-cotorsion pair in an exact category}
A pair of classes $(\class{F},\class{C})$ in an exact category
$\cat{A}$ is a cotorsion pair if the following conditions hold:
\begin{enumerate}
    \item $\Ext^1_{\cat{A}}(F,C) = 0$ for all $F \in \class{F}$ and $C \in
    \class{C}$.

    \item If $\Ext^1_{\cat{A}}(F,X) = 0$ for all $F \in \class{F}$, then $X
    \in \class{C}$.

    \item If $\Ext^1_{\cat{A}}(X,C) = 0$ for all $C \in \class{C}$, then $X \in
    \class{F}$.
\end{enumerate}
We say the cotorsion pair is \emph{hereditary} if $\class{F}$ is closed under taking kernels of epimorphisms between objects of $\class{F}$ and if $\class{C}$ is closed under taking cokernels of monomorphisms between objects of $\class{C}$.
\end{definition}

In particular, letting $\class{P}$ be the class of projective objects in $\class{A}$ and $\class{I}$ be the class of injective objects in $\cat{A}$, one has the projective cotorsion pair
$(\class{P},\class{A})$ and the injective cotorsion pair
$(\class{A},\class{I})$.

A cotorsion pair is said to have \emph{enough projectives} if for
any $X \in \cat{A}$ there is a short exact sequence $C \rightarrowtail F \twoheadrightarrow X $ where $C \in
\class{C}$ and $F \in \class{F}$. We say it has \emph{enough injectives} if it satisfies the dual statement. If both of these hold
we say the cotorsion pair is \emph{complete}. If the cotorsion pair has enough projectives in a way that is functorial with respect to $X$ then we say the cotorsion pair has \emph{enough functorial projectives}. Similarly, we have the terms \emph{enough functorial injectives} and \emph{functorially complete}.

A functor between exact categories is called exact if it preserves short exact sequences. In particular, one can show in the usual way that $\Hom_{\class{A}}(P,-) \mathcolon \class{A} \ar \textbf{Ab}$ is exact if and only if $P$ is projective and the contravariant $\Hom_{\class{A}}(-,I)$ is exact if and only if $I$ is injective. More generally, for any $X \in \class{A}$, the functor $\Hom_{\class{A}}(X,-) \mathcolon \class{A} \ar \textbf{Ab}$, sends short exact sequences to left exact sequences. Similarly the contravariant $\Hom_{\class{A}}(-,X)$ sends right exact sequences to left exact sequences. As in B\"uhler's Remark~12.11 of~\cite{buhler-exact categories} we can define left and right derived functors satisfying the usual properties when the exact category has enough projectives and injectives. In particular, $\Ext^1_{\class{A}}$ above must coincide with the 1st right derived functor of $\Hom_{\class{A}}$ and fit into the usual long exact sequence.

\subsection{Weakly idempotent complete exact categories}\label{subsec-weakly idempotent complete exact categories}

We are interested in the problem of defining model structures on exact categories. The cofibrations will be admissible monomorphisms with certain cokernels and the fibrations will be admissible epimorphisms with certain kernels. Part of the requirement for a model structure is that the cofibrations and fibrations be closed under retracts. So we want admissible monomorphisms and admissible epimorphisms to be closed under retracts. We now show that this is equivalent to insisting that the additive category $\class{A}$ is weakly idempotent complete. The author learned of weakly idempotent complete exact categories from~\cite{buhler-exact categories}.
Except for our main Proposition~\ref{prop-weak idempotent completeness and retracts}, all of the ideas here can be found in Section~7 of~\cite{buhler-exact categories}.

The main idea is very simple. It is not automatic that split monomorphisms (those with left inverses) in exact categories be admissible monomorphisms and that split epimorphisms (those with right inverses) be admissible epimorphisms. This stems from the fact that split monomorphisms in additive categories need not automatically have cokernels, and the dual for split epimorphsms. The concept of weak idempotent completeness removes this problem.

Let $f \mathcolon X \ar Y$ and $g \mathcolon Y \ar X$. Recall that if $gf = 1_X$ then we call $f$ a \emph{split monomorphism}, or a \emph{section} and $g$ a \emph{split epimorphism}, or a \emph{retraction}. A split monomorphism is a monomorphism since it is easily seen to be left cancellable while a split epimorphism is an epimorphism since it is right cancellable. B\"uhler calls a split epimorphism a \emph{retraction} and a split monomorphism a \emph{coretraction} and has a different usage of the word ``section''.

\begin{definition}
We call an additive category $\class{A}$ \emph{weakly idempotent complete} if every split monomorphism has a cokernel and every split epimorphism has a kernel. We call an exact category $(\class{A},\class{E})$ \emph{weakly idempotent complete} if the additive category $\class{A}$ is such.
\end{definition}

For example, the category of free $R$-modules is an exact category which is not weakly idempotent complete whenever there are projective modules which are not free. The ``Eilenberg Swindle'' can be used to see this. The following Proposition summarizes all the results we will need on weakly idempotent complete exact categories. Clear proofs can be found in Section~7 of~\cite{buhler-exact categories}.

\begin{proposition}\label{prop-properties of weakly idempotent complete categories}
Let $(\class{A},\class{E})$ be a weakly idempotent complete exact category. The the following hold:
\begin{enumerate}

\item Every split monomorphism $f \mathcolon X \ar Y$ is an admissible monomorphism. Letting $Y \ar Z$ be its cokernel, the sequence $X \rightarrowtail Y \twoheadrightarrow Z$ is isomorphic to the split exact sequence $X \rightarrowtail X \oplus Z \twoheadrightarrow Z$.

\item Every split epimorphism $g \mathcolon Y \ar Z$ is an admissible epimorphism. Letting $K \ar Y$ be its kernel, the sequence $K \rightarrowtail Y\twoheadrightarrow Z$ is isomorphic to the split exact sequence $K \rightarrowtail K \oplus Z \twoheadrightarrow Z$.

\item If $gf$ is an admissible monomorphism, then $f$ is an admissible monomorphism.

\item If $gf$ is an admissible epimorphism, then $g$ is an admissible epimorphism.
\end{enumerate}
\end{proposition}

The last two properties in the above proposition are easy to remember since they are analogous to standard facts about monomorphisms and epimorphisms in any category.

We now see why we want $(\class{A},\class{E})$ to be weakly idempotent complete when we wish to build an exact model structure on $(\class{A},\class{E})$. We need admissible monomorphisms and admissible epimorphisms to be closed under retracts.

\begin{proposition}\label{prop-weak idempotent completeness and retracts}
Let $(\class{A},\class{E})$ be an exact category. Then the following are equivalent:
\begin{enumerate}
\item $(\class{A},\class{E})$ is weakly idempotent complete.

\item Admissible monomorphisms are closed under retracts.

\item Admissible epimorphisms are closed under retracts.
\end{enumerate}
\end{proposition}

\begin{proof}
(1) $\Rightarrow$ (2). Suppose $(\class{A},\class{E})$ is weakly idempotent complete and suppose we have the commutative diagram below expressing $j \mathcolon A \ar B$ as a retract of $i \mathcolon X \rightarrowtail Y$.
$$\begin{CD}
A    @>f_1>> X  @>g_1>> A \\
@VjVV @VViV  @VVjV \\
B     @>f_2>> Y @>g_2>> B\\
\end{CD}$$
So $g_1f_1 = 1_A$ and $g_2f_2 = 1_B$. In particular, $f_1$ is a split monomorphism. So by Proposition~\ref{prop-properties of weakly idempotent complete categories}, $f_1$ is an admissible monomorphism. By the composition axiom $if_1 = f_2j$ is an admissible monomorphism. So by~\ref{prop-properties of weakly idempotent complete categories} again we get $j$ is an admissible monomorphism. The proof of (1) $\Rightarrow$ (3) is similar to (1) $\Rightarrow$ (2).

(2) $\Rightarrow$ (1). Suppose that the class of admissible monomorphisms is closed under retracts. We wish to show that $(\class{A},\class{E})$ is weakly idempotent complete. By Corollary~7.5 of~\cite{buhler-exact categories} we just need to show that every split monomorphism is an admissible monomorphism. So let $f \mathcolon A \ar C$ be a split monomorphism. So there is a map $g \mathcolon C \ar A$ with $gf = 1_A$. First, we note that the map $A \ar A \oplus B$ defined by the matrix $\binom{1_A}{-f}$ is an admissible monomorphism since the diagram below puts it in a sequence which is isomorphic to the split exact sequence $A \rightarrowtail A \oplus C \twoheadrightarrow C$.
 $$\begin{CD}
A    @> \binom{1_A}{-f} >> A \oplus C  @> (f \ 1_C) >> C \\
@| @VVV  @| \\
A     @>\binom{1_A}{0} >> A \oplus C @> (0  \ 1_C) >> C \\
\end{CD}$$
The middle vertical arrow above is the map $\left( \begin{array}{c}
 1_A \ \ 0 \\
 f \ \ 1_C \\
 \end{array} \right)$ which has the inverse $\left( \begin{array}{c}
 1_A \ \ 0 \\
 -f \ \ 1_C \\
 \end{array} \right)$. Now by hypothesis we will be done if we can show that $f \mathcolon A \ar C$ is a retract of $\binom{1_A}{-f}$. But the commutative diagram below displays that this is indeed the case.

$$\begin{CD}
A    @= A  @= A \\
@VfVV @VV\binom{1_A}{-f}V  @VVfV \\
C     @>\binom{g}{-1_C}>> A \oplus C @> (f \ \ fg-1_C) >> C\\
\end{CD}$$

\end{proof}

\section{Exact model structures}\label{sec-exact model structures}

Hovey defined in~\cite{hovey}, Definition~2.1, the notion of an abelian model structure. This is a model structure on an abelian category in which (trivial) cofibrations are monomorphism with (trivially) cofibrant cokernel and (trivial) fibrations are epimorphisms with (trivially) fibrant kernel. We make a similar definition for model structures on exact categories. First we define some terminology from the theory of model categories.

Note that if $(\class{A},\class{E})$ is an exact category, then for any $X \in \class{A}$, $0 \rightarrowtail X$ is an admissible monomorphism and $X \twoheadrightarrow 0$ is an admissible epimorphism. Now suppose $(\class{A},\class{E})$ has a model structure as defined in Definition~1.1.3 of~\cite{hovey-model-categories}.
\begin{itemize}
\item We say $X$ is \emph{trivial} if $0 \rightarrowtail X$ is a weak equivalence. By the 2 out of 3 axiom and the fact that identity maps are always weak equivalences, this is equivalent to insisting $X \twoheadrightarrow 0$ is a weak equivalence.

\item We say $X$ is \emph{cofibrant} if $0 \rightarrowtail X$ is a cofibration.

\item We say $X$ is \emph{fibrant} if $X \twoheadrightarrow 0$ is a fibration.

\item We say $X$ is \emph{trivially cofibrant} if it is both trivial and cofibrant.

\item We say $X$ is \emph{trivially fibrant} if it is both trivial and fibrant.
\end{itemize}

Recall that in a model category, a \emph{trivial cofibration} is a map which is both a cofibration and a weak equivalence. Thus an object $X$ is trivially cofibrant if and only if $0 \rightarrowtail X$ is a trivial cofibration.  Of course, the analogous is true for trivial fibrations. We are now ready to define an exact model structure.

\begin{definition}\label{def-exact model structures}
Let $(\class{A},\class{E})$ be an exact category. An \emph{exact model structure} on $(\class{A},\class{E})$ is a model structure in the sense of Definition~1.1.3 of~\cite{hovey-model-categories} in which each of the following hold.
\begin{enumerate}
\item A map is a (trivial) cofibration if and only if it is an admissible monomorphism with a (trivially) cofibrant cokernel.
\item A map is a (trivial) fibration if and only if it is an admissible epimorphism with a (trivially) fibrant kernel.
\end{enumerate}
\end{definition}

Note that with any model structure, by the 2 out of 3 axiom, a map $g$ is a weak equivalence if and only if it has a factorization $g = pi$ where $i$ is a trivial cofibration and $p$ is a trivial fibration. So in an exact model structure, a map $g$ is a weak equivalence if and only if it has a factorization $g = pi$ where $i$ is an admissible monomorphism with trivially cofibrant cokernel and $p$ is an admissible epimorphism with trivially fibrant kernel.

\subsection{Hovey's correspondence}

Hovey's correspondence theorem is a translation between model structures on an abelian category $\cat{A}$ and two compatible complete cotorsion pairs on $\cat{A}$. We now point out that this correspondence also works when $\cat{A}$ is a a weakly idempotent complete exact category. We will need the definition of a thick subcategory.

\begin{definition}
Given an exact category $(\class{A},\class{E})$, by a \emph{thick subcategory} of $\cat{A}$ we mean a class of objects $\class{W}$ which is closed under direct summands and such that if two out of three of the terms in a short exact sequence are in $\class{W}$, then so is the third.
\end{definition}

\textbf{Note}: The definition of a thick subcategory usually states that $\class{W}$ be closed under retracts instead of saying direct summands. These are clearly the same if $\class{A}$ is abelian, or weakly idempotent complete. But for a general exact category there is a subtlety. In this case one can show $\class{W}$ is closed under direct summands if and only if given any \emph{admissible} monomorphism $f \mathcolon X \rightarrowtail W$ with $W \in \class{W}$, if $f$ splits then $X \in \class{W}$. [Similar to Remark~7.4 of~\cite{buhler-exact categories}, one can show that such an admissible split monomorphism gives rise to a direct sum decomposition of $W$.]

\begin{theorem}\label{them-Hovey's theorem for exact categories}
Let $(\class{A},\class{E})$ be an exact category with an exact model structure. Let $\class{Q}$ be the class of cofibrant objects, $\class{R}$ the class of fibrant objects and $\class{W}$ the class of trivial objects. Then $\class{W}$ is a thick subcategory of $\cat{A}$ and both $(\class{Q},\class{R} \cap \class{W})$ and $(\class{Q} \cap \class{W} , \class{R})$ are complete cotorsion pairs in $\cat{A}$. If we assume $(\class{A},\class{E})$ is weakly idempotent complete then the converse holds. That is, given two compatible cotorsion pairs $(\class{Q},\class{R} \cap \class{W})$ and $(\class{Q} \cap \class{W} , \class{R})$, each complete and with $\class{W}$ a thick subcategory, then there is an exact model structure on $\cat{A}$ where $\class{Q}$ are the cofibrant objects, $\class{R}$ are the fibrant objects and $\class{W}$ are the trivial objects.
\end{theorem}

\begin{proof}
Hovey proved the statement that $(\class{Q},\class{R} \cap \class{W})$ and $(\class{Q} \cap \class{W} , \class{R})$ are complete cotorsion pairs in~\cite{hovey}, Proposition~4.1, in the abelian case. A careful look at the proof shows that it readily carries over to exact categories. In the second paragraph of that proof, one can show directly that the map $\class{A}(A,X) \ar \class{A}(A,Y)$ is surjective by taking $g \in \class{A}(A,Y)$, forming the pullback of the pair $(g,p)$, and using the splitting to construct the necessary lift. The dual argument of forming a pushout will work for paragraph three.

Hovey proved (again for the abelian case) that $\class{W}$ is a thick subcategory in Lemma~4.3 in~\cite{hovey}. We follow that proof but make necessary modifications for exact categories. First, if $W \in \class{W}$ and $X$ is a direct summand of $W$, then it is easy to see that the map $0 \rightarrowtail X$ is a retract of the map $0 \rightarrowtail W$. Since $0 \rightarrowtail W$ is a weak equivalence the retract axiom tells us $0 \rightarrowtail X$ is also a weak equivalence. This proves that $\class{W}$ is closed under direct summands. Next we wish to prove that if two out of three of the terms in an exact sequence are in $\class{W}$, then so is the third. We will make repeated use of the following observations which follow easily from the 2 out of 3 axiom for model categories: If $X \ar Y$ is a weak equivalence and either $X$ or $Y$ is in $\class{W}$, then so is the other. On the other hand, if $X,Y \in \class{W}$, then \emph{any} map $X \ar Y$ is a weak equivalence.

Now suppose $A \xrightarrow{f} B \xrightarrow{g} C$ is an admissible short exact sequence. Using the factorization axioms, write $g = pi$ where $p$ is a fibration (so an admissible epimorphism with fibrant kernel) and $i$ is a trivial cofibration (so an admissible monomorphism with a trivially cofibrant cokernel). Let $A' = \ker{p}$ and use the universal property of $\ker{p}$ to get the following commutative diagram:
$$\begin{CD}
A    @>f>> B  @>g>> C \\
@VjVV @VViV  @| \\
A'     @>k>> B' @>p>> C\\
\end{CD}$$
By Proposition~2.12 of~\cite{buhler-exact categories} we get that the left square above is a pushout square. Letting $q = \cok{i}$, one can then prove directly that  $qk$ is the cokernel of $j$. (This is an easy exercise, first use the universal property of the pushout square followed by the universal property of the cokernel $q$.) Since $i$ and $f$ are admissible monomorphisms, it follows that the composition $kj = if$ is an admissible monomorphism. Furthermore $k$ is an admissible monomorphism and $j$ was shown above to have a cokernel. It follows from Quillen's ``obscure axiom'', Proposition~2.16 of~\cite{buhler-exact categories} that $j$ is also an admissible monomorphism. Since the target of $q$ is trivially cofibrant we see that the target of $qk$ is also trivially cofibrant. Therefore $j$ is a trivial cofibration. In particular $j$ is a weak equivalence.

Now if $A \in \class{W}$, then $A' \in \class{W}$ since $j$ is a weak equivalence. In this case $p$ is a trivial fibration. So then $g = pi$ is a weak equivalence. So in this case $B \in \class{W}$ if and only if $C \in \class{W}$.

On the other hand, if $B,C \in \class{W}$, then any map $B \ar C$ is a weak equivalence and in particular $g = pi$ is a weak equivalence. In this case, by the 2 out of 3 axiom, $p$ must be a trivial fibration. Therefore $A' \in \class{W}$. Since we proved above that $j$ is a weak equivalence it follows that $A \in \class{W}$.

Hovey proved the converse statement in Section~5 of~\cite{hovey} in the case of abelian categories. That proof can be adapted here and we point out now how the weakly idempotent completeness hypothesis is necessary. First, in light of Proposition~\ref{prop-weak idempotent completeness and retracts}, we need the weakly idempotent complete hypothesis to prove the retract axiom. Also, throughout Hovey's proof, there are several uses of constructing Kernel-Cokernel exact sequences and uses of the Snake Lemma. These theorems hold in exact categories with the weakly idempotent complete hypothesis by Proposition~8.11 and Corollary~8.13 of~\cite{buhler-exact categories}. In each application one can check that the proper morphisms are ``admissible'' in the sense of Definition~8.1 of~\cite{buhler-exact categories}.
\end{proof}

\begin{corollary}\label{cor-Hovey's theorem for exact categories}
Let $(\class{A},\class{E})$ be a weakly idempotent complete exact category. Then there is a one-to-one correspondence between exact model structures on $\cat{A}$ and complete cotorsion pairs $(\class{Q},\class{R} \cap \class{W})$ and $(\class{Q} \cap \class{W} , \class{R})$ where $\class{W}$ is a thick subcategory of $\cat{A}$. Given a model structure, $\class{Q}$ is the class of cofibrant objects, $\class{R}$ the class of fibrant objects and $\class{W}$ the class of trivial objects. Conversely, given the compatible cotorsion pairs with $\class{W}$ thick, a cofibration (resp. trivial cofibration) is an admissible monomorphism with a cokernel in $\class{Q}$ (resp. $\class{Q} \cap \class{W}$), and a fibration (reps. trivial fibration) is an admissible epimorphism with a kernel in $\class{R}$ (resp. $\class{R} \cap \class{W}$). The weak equivalences are then the maps $g$ which factor as $g = pi$ where $i$ is a trivial cofibration and $p$ is a trivial fibration.
\end{corollary}

\begin{remark}
Functorial factorizations in an exact model structure correspond directly to functorial completeness (as defined after Definition~\ref{def-cotorsion pair in an exact category}) of each of the corresponding cotorsion pairs.
\end{remark}

\section{The homotopy category of an exact model structure}\label{sec-the homotopy category of an exact model structure}

Let $\class{M}$ be a model category and $X,Y \in \class{M}$. The axioms of a model category allow for the construction of homotopy relations on the set $\Hom_{\class{M}}(X,Y)$. The first is \emph{left homotopy} defined in terms of ``cylinder objects''. The second is \emph{right homotopy} defined in terms of ``path objects''. It is interesting to note that the standard approach to defining left and right homotopy (for example, see Sections~4 and~5 of~\cite{dwyer-spalinski}, or Section~1.2 of~\cite{hovey-model-categories}) and for proving all of the basic facts and theorems on homotopy including the ``Fundamental Theorem of Model Categories'' do not require the full assumption that $\class{M}$ has all finite limits and colimits. All that is needed are terminal and initial objects, the existence of finite ``self-products'' $X \prod X$ and ``self-coproducts'' $X \coprod X$, and the \emph{existence} of any pullback of any fibration and any pushout of any cofibration. The fact that pullbacks of fibrations are again fibrations and that pushouts of cofibrations are again cofibrations follows as usual. Any exact model structure clearly has these properties: First, since an exact category is additive it contains the initial/terminal object 0 as well as all finite biproducts. Second, the definition of an exact model structure requires all cofibrations (resp. fibrations) to be admissible monomorphisms (resp. admissible epimorphisms) and the axioms for an exact category require pushouts (resp. pullbacks) of all admissible monomorphisms (resp. admissible epimorphisms) to exist.

Our intention now is to remind the reader of some results and notation from homotopy theory which we will use ahead. We will state these as facts about exact model structures when really they hold for all model structures on categories having the above mentioned limits and colimits. Our notation follows sections~4 and ~5 of~\cite{dwyer-spalinski} here.

\begin{fact}\label{fact-basic facts on homotopy and homotopy category}
Let $(\class{A},\class{E})$ be an exact category with an exact model structure and let $X,Y \in \class{A}$. The following standard facts of model categories all hold, with the usual proofs, without any further assumptions on the existence of limits and colimits:

\begin{enumerate}
\item There is a relation called \emph{left homotopy}, denoted $\sim^l$, on $\Hom_{\class{A}}(X,Y)$ which is an equivalence relation whenever $X$ is cofibrant.
\item There is a relation called \emph{right homotopy}, denoted $\sim^r$, on $\Hom_{\class{A}}(X,Y)$ which is an equivalence relation whenever $Y$ is fibrant.
\item\label{item-equality of left and right homotopy} Whenever $X$ and $Y$ are each cofibrant and fibrant then $\sim^l \ = \ \sim^r$. In this case we simply call the relation \emph{homotopy} and denote it by $\sim$. We then also define $\pi(X,Y) := \Hom_{\class{A}}(X,Y)/\sim$.
\item Denote the full subcategory of cofibrant and fibrant subobjects by $\class{A}_{c,f}$. Then the homotopy relation $\sim$ is compatible with composition so that we get a homotopy category $\pi\class{A}_{c,f}$ where the $\Hom$ sets are as
    in item ~\ref{item-equality of left and right homotopy} above.
\item Any $f \mathcolon X \ar Y$ in $\class{A}_{c,f}$ which is an isomorphism in $\pi\class{A}_{c,f}$ is called a \emph{homotopy equivalence}. So $f$ is a homotopy equivalence if there exists a map $g \mathcolon Y \ar X$ such that $gf \sim 1_X$ and $fg \sim 1_Y$. A map $f \mathcolon X \ar Y$ in $\class{A}_{c,f}$ is a homotopy equivalence if and only if it is a weak equivalence in $\class{A}$.
\item Given any $A \in \cat{A}$, it has a cofibrant replacement $p_X \mathcolon QX \twoheadrightarrow X$ and a fibrant replacement $i_X \mathcolon X \rightarrowtail RX$. $QX$ is cofibrant and $p_X$ is a trivial fibration, and $RX$ is fibrant and $i_X$ is a trivial cofibration. We can insist these exist functorially if one wishes.
\item We can define the homotopy category $\text{Ho}(\cat{A})$ to be the category with the same objects as $\cat{A}$ and with $\Hom_{\text{Ho}(\cat{A})}(A,B) = \pi(RQA,RQB)$. There is a canonical functor $\gamma_{\cat{A}} \mathcolon \cat{A} \ar \text{Ho}(\cat{A})$ which sends a map $f \mathcolon A \ar B$ to the homotopy class $[f'] \in \pi(RQA,RQB)$, where below $\tilde{f}$ and $f'$ are \emph{any} maps making the diagrams commute:
    $$\begin{CD}
RQA    @>f'>> RQB \\
@A i_{QA} AA @AA i_{QB} A  \\
QA     @> \tilde{f} >> QB \\
@Vp_A VV @VVp_B V  \\
A     @>f>> B.
\end{CD}$$ On the other hand, given \emph{any} map $f \mathcolon A \ar B$ in the homotopy category $\text{Ho}(\cat{A})$, there is a map $f' \mathcolon RQA \ar RQB$ in $\cat{A}$, unique up to homotopy, such that $$f = \gamma_{\cat{A}}(p_B) \circ (\gamma_{\cat{A}}(i_{QB}))^{-1} \circ \gamma_{\cat{A}}(f') \circ \gamma_{\cat{A}}(i_{QA}) \circ (\gamma_{\cat{A}}(p_A))^{-1}.$$

\item The functor $\gamma_{\cat{A}}$ sends weak equivalences in $\cat{A}$ to isomorphisms in $\text{Ho}(\cat{A})$ and in fact $\text{Ho}(\cat{A})$ as a localization of $\cat{A}$ with respect to the weak equivalences.
\end{enumerate}
\end{fact}

Let $(\cat{A}_1,\class{E}_1)$ and $(\cat{A}_2,\class{E}_2)$ be exact model structures. We call a functor $F \mathcolon \cat{A}_1 \ar \cat{A}_2$ which preserves weak equivalences \emph{homotopical}.

\begin{lemma}\label{lemma-definition of HoF for a homotopical functor F}
Any homotopical functor of exact model structures $F \mathcolon \cat{A}_1 \ar \cat{A}_2$ induces a unique functor $\text{Ho}F \mathcolon \text{Ho}\cat{A}_1 \ar \text{Ho}\cat{A}_2$ making a commutative square
$$\begin{CD}
\cat{A}_1    @>F>> \cat{A}_2 \\
@V\gamma_{\cat{A}_1} VV  @VV \gamma_{\cat{A}_2} V \\
\text{Ho}\cat{A}_1 @> \text{Ho}F >> \text{Ho}\cat{A}_2 .\\
\end{CD}$$ On objects $\text{Ho}F$ is the identity and for an arrow $f \in \text{Ho}\cat{A}_1(A,B)$ represented by $[f'] \in \pi(RQA,RQB) := \Hom_{\class{A}_1}(RQA,RQB)/\sim$ we have $$\text{Ho}F(f) = \gamma_{\cat{A}_2}F(p_B) \circ (\gamma_{\cat{A}_2}F(i_{QB}))^{-1} \circ \gamma_{\cat{A}_2}F(f') \circ \gamma_{\cat{A}_2}F(i_{QA}) \circ (\gamma_{\cat{A}_2}F(p_A))^{-1}.$$
\end{lemma}

\begin{proof}
By Proposition~5.8 of~\cite{dwyer-spalinski} the functor $\gamma_{\cat{A}_2}$ sends weak equivalences to isomorphisms and so the composite $\gamma_{\cat{A}_2} \circ F$ sends weak equivalences to isomorphisms. By Theorem~6.2 of~\cite{dwyer-spalinski}, $\text{Ho}\cat{A}_1$ is a localization of $\cat{A}_1$ with respect to its weak equivalences which means there exists a unique functor $\text{Ho}F \mathcolon \text{Ho}\cat{A}_1 \ar \text{Ho}\cat{A}_2$ such that $\text{Ho}F \circ \gamma_{\cat{A}_1}  = \gamma_{\cat{A}_2} \circ F$. It follows from the proof of Theorem~6.2 of~\cite{dwyer-spalinski} that $\text{Ho}F$ must be defined exactly as stated.
\end{proof}

\subsection{Homotopic maps in exact model structures}

We now give a nice characterization of left and right homotopic maps in exact model structures. The statement and proof of the following Proposition is inspired by Proposition~9.1 of~\cite{hovey}.

\begin{proposition}\label{prop-left and right homotopic maps in exact model structures}
Assume $(\class{A},\class{E})$ is an exact category with an exact model structure. Let $(\class{Q},\class{R} \cap \class{W})$ and $(\class{Q} \cap \class{W} , \class{R})$ be the corresponding complete cotorsion pairs of Theorem~\ref{them-Hovey's theorem for exact categories}.
\begin{enumerate}
\item Two maps $f,g \mathcolon X \ar Y$ in $\class{A}$ are right homotopic if and only if $g-f$ factors through an object of $\class{Q} \cap \class{W}$.
\item Two maps $f,g \mathcolon X \ar Y$ in $\class{A}$ are left homotopic if and only if $g-f$ factors through an object of $\class{R} \cap \class{W}$.
\item Suppose $Y$ is fibrant. Then two maps $f,g \mathcolon X \ar Y$ in $\class{A}$ are right homotopic if and only if $g-f$ factors through an object of $\class{Q} \cap \class{R} \cap \class{W}$.
\item Suppose $X$ is cofibrant. Then two maps $f,g \mathcolon X \ar Y$ in $\class{A}$ are left homotopic if and only if $g-f$ factors through an object of $\class{Q} \cap \class{R} \cap \class{W}$.
 \item Suppose $X$ is cofibrant and $Y$ is fibrant. Then two maps $f,g \mathcolon X \ar Y$ in $\class{A}$ are homotopic if and only if $g-f$ factors through an object of $\class{Q} \cap \class{R} \cap \class{W}$ if and only if $g-f$ factors through an object of $\class{Q} \cap \class{W}$ if and only if $g-f$ factors through an object of $\class{R} \cap \class{W}$.
\end{enumerate}
\end{proposition}

\begin{proof}
The first two statements are dual and we will prove (1). We want to characterize right homotopy, so we first need to construct a good path object of $Y$, which is a factorization of the diagonal map $\Delta \mathcolon Y \ar Y \oplus Y$ as a weak equivalence followed by a fibration. Since $(\class{Q} \cap \class{W} , \class{R})$ is a complete cotorsion pair in $\cat{A}$ we can find an admissible epimorphism $q \mathcolon Q \twoheadrightarrow Y$ where $Q \in \class{Q} \cap \class{W}$ and with $\ker{q} = R \in \class{R}$.
Now consider the factorization $Y \xrightarrow{i} Y \oplus Q \xrightarrow{p} Y \oplus Y$ where in the standard matrix notation $i = \binom{1_Y}{0}$ is the canonical inclusion and $p = \left( \begin{array}{c}
 1_Y \ \ 0 \\
 1_Y \ \ q \\
 \end{array} \right)$. Clearly $\Delta = pi$ and $i$ is a trivial cofibration since it is an admissible monomorphism with cokernel $Q \in \class{Q} \cap \class{W}$. We claim that $p$ is a fibration, that is, it is an admissible epimorphism with kernel in $\class{R}$. Indeed, one can check that $\ker{p} = \ker{q} = R \in \class{R}$. So we now show that $p$ is an admissible epimorphism. Since $q$ is an admissible epimorphism it follows from Proposition~2.9 of~\cite{buhler-exact categories} that $q \oplus q = \left( \begin{array}{c}
 q \ \ 0 \\
 0 \ \ q \\
 \end{array} \right)$ is an admissible epimorphism. But also,
 $$\left( \begin{array}{c}
 1_Y \ \ 0 \\
 1_Y \ \ q \\
 \end{array} \right) \left( \begin{array}{c}
 q \ \ \ \ 0 \\
 -1_Q \ \ 1_Q \\
 \end{array} \right) = \left( \begin{array}{c}
 q \ \ 0 \\
 0 \ \ q \\
 \end{array} \right)$$ where $\left( \begin{array}{c}
 q \ \ \ \ 0 \\
 -1_Q \ \ 1_Q \\
 \end{array} \right)$ is the map $Q \oplus Q \ar Y \oplus Q$. It follows from Proposition~\ref{prop-properties of weakly idempotent complete categories} that $p$ is an admissible epimorphism. So now we have proved that $Y \oplus Q$ is a good path object for $Y$.

Now let $f,g \mathcolon X \ar Y$ be two maps in $\class{A}$. Then by definition, $f \sim^r g$ if and only if there is a map $X \xrightarrow{(\alpha  \ \beta)} Y \oplus Q$ such that $(f  \ g) = p \, (\alpha \ \beta)$. That is, $f \sim^r g$ if and only if there are maps $\alpha \mathcolon X \ar Y$ and $\beta \mathcolon X \ar Q$ such that $f = \alpha$ and $g = \alpha + q \beta$. So if and only if $g = f + q \beta$ for some $\beta \mathcolon X \ar Q$. So if and only if $g-f = q\beta$ for some $\beta \mathcolon X \ar Q$. In particular, this proves that $f \sim^r g$ if and only if $g - f$ factors through the $Q$ via $q$. On the other hand, say $g - f$ factors through some object $Q' \in \class{Q} \cap \class{W}$. Then since $q  \mathcolon Q \twoheadrightarrow Y$ is a fibration, the map $Q' \ar Y$ lifts over $q$. Thus the factorization of $g -f$ through $Q'$ extends to a factorization of $g-f$ through $q$. Therefore, $f \sim^r g$. This completes the proof of (1).

Statements (3) and (4) are dual and we will prove (4). Suppose that $X$ is cofibrant and $f , g \mathcolon X \ar Y$ are left homotopic. Then by (2) we know $g-f$ factors through some object $R \in \class{R} \cap \class{W}$. Let $Q(R) \twoheadrightarrow R$ be a cofibrant replacement of $R$. Then since $X$ is cofibrant, the map $X \ar R$ in the factorization of $g-f$, lifts over $Q(R) \twoheadrightarrow R$. Thus the factorization of $g-f$ extends to a factorization of $g-f$ through $Q(R)$. Since $Q(R) \in \class{Q} \cap \class{R} \cap \class{W}$ we are done.

Now (5) is true since left and right homotopy coincide when $X$ is cofibrant and $Y$ is fibrant.
\end{proof}

\begin{remark}
By Proposition~\ref{prop-left and right homotopic maps in exact model structures} it is immediate that when $X$ is cofibrant then left homotopic implies right homotopic and the dual is true when $Y$ is fibrant. This of course is a special case of the general fact that holds in all model categories.
\end{remark}

\begin{example}
Consider the flat model structure on chain complexes of $R$-modules constructed in~\cite{gillespie}, or chain complexes sheaves~\cite{gillespie-sheaves}, or chain complexes of quasi-coherent sheaves on a quasi-compact and semi-separated scheme~\cite{gillespie-quasi-coherent}. With this model structure Proposition~\ref{prop-left and right homotopic maps in exact model structures} says two chain maps $f,g \mathcolon X \ar Y$ are right homotopic if and only if their difference factors through a flat complex and are left homotopic if and only if their difference factors through a cotorsion complex. If $X$ is a dg-flat complex and $Y$ is a dg-cotorsion complex, then $f \sim g$ if and only if their difference factors through a flat cotorsion complex.
\end{example}

\subsection{Projective, injective and Frobenius model structures}

Assume $(\class{A},\class{E})$ is an exact category with an exact model structure. We now look at what we call projective, injective, and Frobenius model structures.

\begin{definition}
Assume $(\class{A},\class{E})$ is an exact category with an exact model structure.
\begin{enumerate}
\item We call the model structure on $\cat{A}$ \emph{projective} if the trivially cofibrant objects coincide with the projectives.
\item We call the model structure on $\cat{A}$ \emph{injective} if the trivially fibrant objects coincide with the injectives.
\item We call the model structure on $\cat{A}$ \emph{Frobenius} if the trivially cofibrant objects coincide with the projectives and the trivially fibrant objects coincide with the injectives.
\end{enumerate}
\end{definition}

\begin{lemma}\label{lemma-projective and injective model structures}
Assume $(\class{A},\class{E})$ is an exact category with an exact model structure.
\begin{enumerate}
\item It is a projective model structure if and only if every object is fibrant.
\item It is an injective model structure if and only if every object is cofibrant.
\end{enumerate}
\end{lemma}

\begin{proof}
This is automatic when looking at the corresponding cotorsion pairs $(\class{Q},\class{R} \cap \class{W})$ and $(\class{Q} \cap \class{W} , \class{R})$ in $\cat{A}$. For example, the model structure is projective if and only if $\class{Q} \cap \class{W}$ is the class of projectives if and only if every object is in $\class{R}$.
\end{proof}

\begin{lemma}\label{lemma-characterizations of Frobenius model strucs}
Assume $(\class{A},\class{E})$ is an exact category with an exact model structure. Then the following are equivalent.
\begin{enumerate}
\item The model structure is Frobenius.
\item The model structure is both projective and injective.
\item Every object is both cofibrant and fibrant.
\item The corresponding cotorsion pairs are $(\class{A},\class{W})$ and $(\class{W},\class{A})$.
\item The projectives and injectives coincide and form the class $\class{W}$ of trivial objects. We often call an object in $\class{W}$ \emph{pro-injective}.
\end{enumerate}
\end{lemma}

\begin{proof}
$(1) \Rightarrow (2)$ by definition. $(2) \Rightarrow (3)$ by Lemma~\ref{lemma-projective and injective model structures}. For $(3) \Rightarrow (4)$ note that the cotorsion pairs $(\class{Q},\class{R} \cap \class{W})$ and $(\class{Q} \cap \class{W} , \class{R})$ must collapse to $(\class{A},\class{W})$ and $(\class{W},\class{A})$ and this implies (5) that $\class{W}$ must be the class of pro-injective objects.

$(5) \Rightarrow (1)$ Suppose $(\class{Q},\class{R} \cap \class{W})$ and $(\class{Q} \cap \class{W} , \class{R})$ are the corresponding cotorsion pairs where the trivial objects $\class{W}$ are the class of pro-injectives. Then it follows that $\class{W} = \class{Q} \cap \class{W} = \class{R} \cap \class{W}$ since the right side of a cotorsion pairs always contains the injectives while the left side always contains the projectives. So the model structure is Frobenius.
\end{proof}

\begin{corollary}\label{cor-characterizations of homotopic maps in inj. proj. Frob. model structures}
Assume $(\class{A},\class{E})$ is an exact category with an exact model structure.
\begin{enumerate}
\item Suppose $\cat{A}$ has a projective model structure. Then two maps $f,g \mathcolon X \ar Y$ in $\class{A}$ are right homotopic if and only if $g-f$ factors through a projective object. They are left homotopic if and only if $g-f$ factors through a trivial object. In particular, when $X$ is cofibrant $f$ and $g$ are homotopic if and only if $g-f$ factors through a projective if and only if $g-f$ factors through a trivial object.
\item Suppose $\cat{A}$ has an injective model structure. Then two maps $f,g \mathcolon X \ar Y$ in $\class{A}$ are left homotopic if and only if $g-f$ factors through an injective object. They are right homotopic if and only if $g-f$ factors through a trivial object. In particular, when $Y$ is fibrant $f$ and $g$ are homotopic if and only if $g-f$ factors through a injective if and only if $g-f$ factors through a trivial object.
\item Suppose $\cat{A}$ has a Frobenius model structure. Then two maps $f,g \mathcolon X \ar Y$ in $\class{A}$ are homotopic if and only if $g-f$ factors through a pro-injective object.
\end{enumerate}
\end{corollary}

\begin{proof}
This follows from Proposition~\ref{prop-left and right homotopic maps in exact model structures}.
\end{proof}

\section{Applications and Examples}\label{sec-examples and applications}

There are many examples of exact model structures and in particular injective, projective and Frobenius ones. These include the usual projective and injective model structures on chain complexes of $R$-modules, and the model structure on modules over a Frobenius ring $R$. Each of these are described nicely in~\cite{hovey-model-categories}. There are also both injective and projective model structures on the categories of modules over Gorenstein rings, or more generally Ding-Chen rings, which are discussed in detail in~\cite{hovey},\cite{gillespie-hovey} and~\cite{gillespie-model strucs on modules over Ding-Chen rings}. Our intention in this section is to look at two other interesting occurrences of these types of model structures. First, we see the general fact that any \emph{hereditary} exact model structure contains exact sub-model structures that are projective, injective and Frobenius. The second is the Frobenius model structure describing the classical homotopy category of complexes $\cat{K}(R)$.

\subsection{Sub-model structures of hereditary exact model structures}

In the standard approach to constructing the homotopy category $\text{Ho}\cat{M}$ of a model category $\cat{M}$ we generally consider the equivalent localization categories $\text{Ho}\cat{M}_c$, $\text{Ho}\cat{M}_f$, and $\text{Ho}\cat{M}_{c,f}$, and use the isomorphism $\text{Ho}\cat{M}_{c,f} \cong \cat{M}_{c,f}/\sim$ where $\cat{M}_{c,f}/\sim$ is the classical homotopy category. But of course model structures are meant to describe localization categories and homotopy, so it is natural ask if there are model structures on $\cat{M}_c$, $\cat{M}_f$, or $\cat{M}_{c,f}$. There are very natural model structures describing these localizations when $\cat{A}$ is an hereditary exact model structure on a weakly idempotent complete category $(\cat{A},\class{E})$, in particular when $\cat{A}$ is an hereditary abelian model structure. Of course, all of the homotopy categories with respect to these model structures are equivalent categories. But our point is that $\cat{A}_c$, $\cat{A}_f$, and $\cat{A}_{c,f}$ can each be thought of as exact sub-model structures of $\cat{A}$. Moreover, these sub-model structures are respectively injective, projective and Frobenius.

Throughout this section, we assume $\cat{A}$ is a weakly idempotent complete exact category and that $\cat{A}$ has an hereditary exact model structure. By \emph{hereditary} we mean that the corresponding cotorsion pairs $(\class{Q},\class{R} \cap \class{W})$ and $(\class{Q} \cap \class{W} , \class{R})$ are each hereditary. Let $\cat{A}_f$ be the full subcategory of fibrant objects and let $\cat{A}_c$ be the full subcategory of cofibrant objects and $\cat{A}_{c,f}$ the full subcategory of cofibrant-fibrant objects.

It is easy to see that if $\class{S}$ is any full additive subcategory of $\class{A}$ and if $\class{S}$ is closed under extensions then $\class{S}$ becomes an exact category where the class of short exact sequences is taken to be all short exact sequences in $\cat{A}$ in which all three terms are objects from $\class{S}$. A subcategory such as $\cat{S}$ equipped with these short exact sequences is called a \emph{fully exact subcategory} of $\cat{A}$. In this way $\cat{A}_f$, $\cat{A}_c$ and $\cat{A}_{c,f}$ are each fully exact subcategories of $\cat{A}$. They are also each closed under direct summands and so the following lemma tells us that the inherited exact structures are each weakly idempotent complete.

\begin{lemma}\label{lemma-full subcategories closed under retracts are weakly idempotent complete}
Let $\class{S}$ be a fully exact subcategory of the weakly idempotent complete exact category $\class{A}$. If $\class{S}$ is closed under direct summands then $\class{S}$ is also weakly idempotent complete.
\end{lemma}

\begin{proof}
Suppose that $\class{S}$ is closed under direct summands and let $i \mathcolon X \ar Y$ be a split monomorphism in $\class{S}$. We want to see that $i$ has a cokernel in $\class{S}$. Viewing $i$ as a split monomorphism in $\class{A}$ we know it has a cokernel $Y \ar Z$ and from Proposition~\ref{prop-properties of weakly idempotent complete categories} we know that $X \rightarrowtail Y \twoheadrightarrow Z$ is isomorphic to the short exact sequence $X \rightarrowtail X \oplus Z \twoheadrightarrow Z$. Since $Y \cong X \oplus Z$ is in $\class{S}$, the direct summand $Z \in \class{S}$ and in particular the cokernel of $i$ is in $\class{S}$. So $\class{S}$ is weakly idempotent complete.
\end{proof}

We fix the following notation and definitions on the given category $\cat{A}$:

$\class{W}_f$ is the class of trivial objects in $\cat{A}_f$. So $\class{W}_f = \class{W} \cap \class{R}$.

$\class{W}_c$ is the class of trivial objects in $\cat{A}_c$. So $\class{W}_c = \class{W} \cap \class{Q}$.

$\class{W}_{c,f}$ is the class of trivial objects in $\cat{A}_{c,f}$. So $\class{W}_{c,f} = \class{W} \cap \class{Q} \cap \class{R}$.

$\class{Q}_f$ is the class of cofibrant objects in $\cat{A}_f$. So $\class{Q}_f = \class{Q} \cap \class{R}$.

$\class{Q}_c$ is the class of cofibrant objects in $\cat{A}_c$. So $\class{Q}_c = \class{Q}$.

$\class{Q}_{c,f}$ is the class of cofibrant objects in $\cat{A}_{c,f}$. So $\class{Q}_{c,f} = \class{Q} \cap \class{R}$.

$\class{R}_f$ is the class of fibrant objects in $\cat{A}_f$. So $\class{R}_f = \class{R}$.

$\class{R}_c$ is the class of fibrant objects in $\cat{A}_c$. So $\class{R}_c = \class{R} \cap \class{Q}$.

$\class{R}_{c,f}$ is the class of fibrant objects in $\cat{A}_{c,f}$. So $\class{R}_{c,f} = \class{Q} \cap \class{R}$.

\begin{proposition}\label{prop-submodel structures of an abelian model cat}
Let $\cat{A}$ be a weakly idempotent complete exact category with an exact model structure. Suppose the associated complete cotorsion pairs are $(\class{Q},\class{R} \cap \class{W})$ and $(\class{Q} \cap \class{W} , \class{R})$ are each hereditary. Then the following each hold:
\begin{enumerate}
\item Each of $\cat{A}_c$, $\cat{A}_f$ and $\class{A}_{c,f}$ are weakly idempotent complete exact subcategories of $\class{A}$ under the exact structure naturally inherited from $\class{A}$.

\item $(\class{Q}_f,\class{R}_f \cap \class{W}_f)$ and $(\class{Q}_f \cap \class{W}_f , \class{R}_f)$ are both complete hereditary cotorsion pairs in $\cat{A}_f$. The resulting exact model structure on $\cat{A}_f$ is a projective model structure.

\item $(\class{Q}_c,\class{R}_c \cap \class{W}_c)$ and $(\class{Q}_c \cap \class{W}_c , \class{R}_c)$ are both complete hereditary cotorsion pairs in $\cat{A}_c$. The resulting exact model structure on $\cat{A}_c$ is an injective model structure.

\item $(\class{Q}_{c,f},\class{R}_{c,f} \cap \class{W}_{c,f})$ and $(\class{Q}_{c,f} \cap \class{W}_{c,f} , \class{R}_{c,f})$ are both complete hereditary cotorsion pairs on $\cat{A}_{c,f}$. The resulting exact model structure on $\cat{A}_{c,f}$ is a Frobenius model structure.
\end{enumerate}

\end{proposition}

\begin{proof}This first statement comes from Lemma~\ref{lemma-full subcategories closed under retracts are weakly idempotent complete}.
We will show that $(\class{Q}_f,\class{R}_f \cap \class{W}_f)$ is a complete cotorsion pair in $\cat{A}_f$. Proving that the others are complete cotorsion pairs is similar.

First say $Q \in \class{Q}_f = \class{Q} \cap \class{R}$ and $R \in \class{R}_f \cap \class{W}_f = \class{R} \cap \class{W}$. Then since $(\class{Q},\class{R} \cap \class{W})$ is a cotorsion pair in $\class{A}$, any short exact sequence $R \rightarrowtail Z \twoheadrightarrow Q$ in $\cat{A}_f$ must split. So $\Ext_{\cat{A}_f}(Q,R) = 0$.

Next, say $X \in \cat{A}_f$ and $\Ext_{\cat{A}_f}(X,R) = 0$ for all $R \in \class{R}_f \cap \class{W}_f$. We want to show $X \in \class{Q}_f$ and this requires showing $X \in \class{Q}$. To show $X \in \class{Q}$ just let $R \in \class{R} \cap \class{W}$ be arbitrary and argue that $\Ext_{\cat{A}}(X,R) = 0$. But this is clearly true by hypothesis since any short exact sequence $R \rightarrowtail Z \twoheadrightarrow X$ in $\Ext_{\cat{A}}(X,R)$ is an element of $\Ext_{\cat{A}_f}(X,R) = 0$.

Lastly, say $X \in \cat{A}_f$ and $\Ext_{\cat{A}_f}(Q,X) = 0$ for all $Q \in \class{Q}_f$. We want to show $X \in \class{R}_f \cap \class{W}_f = \class{R} \cap \class{W}$. Since $(\class{Q},\class{R} \cap \class{W})$ is a complete cotorsion pair in $\cat{A}$ we get a short exact sequence $X  \rightarrowtail R \twoheadrightarrow Q $ where $R \in \class{R} \cap \class{W}$ and $Q \in \class{Q}$. Since $X,R$ are each in the coresolving class $\class{R}$ we get $Q \in \class{R}$. This means that $Q \in \class{Q}_f$ and $X \rightarrowtail R \twoheadrightarrow Q$ is an element of $\Ext_{\cat{A}_f}(Q,X) = 0$. Therefore the sequence must split and so $X$ is a direct summand of $R$. It follows that $X \in \class{R} \cap \class{W}$.

Now the model structure on $\cat{A}_f$ determined by $(\class{Q}_f,\class{R}_f \cap \class{W}_f)$ and $(\class{Q}_f \cap \class{W}_f , \class{R}_f)$ must be a projective model structure since each object is fibrant. The class $\class{Q}_f \cap \class{W}_f$ are the projective objects in $\cat{A}_f$. Similarly, the model structure on $\cat{A}_c$ is injective since each object is cofibrant and $\class{R}_c \cap \class{W}_c$ are the injective objects. The model structure on $\cat{A}_{c,f}$ is Frobenius because every object is both cofibrant and fibrant and $\class{W}_{c,f}$ are the pro-injective objects.
\end{proof}

\begin{definition}
Let $\cat{A}_0 \subseteq \cat{A}$ be a fully exact subcategory with an exact model structure. We call $\cat{A}_0$ a \emph{full equivalent sub-model structure} if the inclusion functor $i \mathcolon \cat{A}_0 \ar \cat{A}$ preserves the model structure and if the induced functor $\text{Ho}(i) \mathcolon \text{Ho}\cat{A}_0 \ar \text{Ho}\cat{A}$ displays $\text{Ho}\cat{A}_0$ as a full equivalent subcategory of $\text{Ho}\cat{A}$.
\end{definition}

\begin{corollary}\label{prop-commutative diagram of homotopy categories}
Let $\cat{A}$ be a weakly idempotent complete exact category with an hereditary exact model structure. Then with the model structure from Proposition~\ref{prop-submodel structures of an abelian model cat}, $\cat{A}_c$, $\cat{A}_f$, and $\cat{A}_{c,f}$ are full equivalent sub-model structures of $\cat{A}$. That is, each functor $\text{Ho}(i)$ in the induced commutative diagram below is an inclusion and displays each source category as a full equivalent subcategory of the corresponding target category:
$$\begin{CD}
\text{Ho}(\cat{A}_{c,f})    @>\text{Ho}i>> \text{Ho}(\cat{A}_f) \\
@V \text{Ho}i VV  @VV \text{Ho}i V \\
\text{Ho}(\cat{A}_c) @>\text{Ho}i>> \text{Ho}(\cat{A}). \\
\end{CD}$$
\end{corollary}

\begin{proof}
Each inclusion functor $i$ is clearly model structure preserving. So if we take $F$ to be one of the inclusion functors $i$, and follow the rule stated for $\text{Ho}F$ in Lemma~\ref{lemma-definition of HoF for a homotopical functor F} we see that for an arrow $f = [f']$, we have $\text{Ho}i(f) = [i(f')] = [f']$ is the identity on arrows. So the homotopy categories sit properly as full subcategories. To see that each $\text{Ho}i$ is a full equivalence we use Proposition~IV.4.2 of~\cite{maclane-cat work}. For example, the functor $\text{Ho}i \mathcolon \text{Ho}(\cat{A}_{c}) \ar \text{Ho}(\cat{A})$ is an equivalence since given any object $A \in \text{Ho}(\cat{A})$ we have the isomorphism $\gamma_{\cat{A}}(p_X) \mathcolon QX \ar X$ in $\text{Ho}(\cat{A})$. In particular $Q$ induces a the inverse equivalence $\text{Ho}Q \mathcolon \text{Ho}(\cat{A}) \ar \text{Ho}(\cat{A}_c)$.
\end{proof}

\begin{remark}
It follows from Propositions~\ref{prop-left and right homotopic maps in exact model structures} and~\ref{prop-submodel structures of an abelian model cat} that the left and right homotopy relations are unambiguous when considered in either $\class{A}$ or one of its sub-model structures $\class{A}_c$, $\class{A}_f$ and $\class{A}_{c,f}$. For example, two maps in $\class{A}_c$, are left (resp. right) homotopic if and only if they are left (resp. right) homotopic in $\class{A}$. In particular, the notation $\pi\cat{A}_{c,f} := \cat{A}_{c,f}/\sim$ for the homotopy category is unambiguous.
\end{remark}

\subsection{Classical homotopy theory of chain complexes}

Let $R$ be a ring. We now show the details to a very simple construction of the model structure on $\text{Ch}(R)$ describing the classical homotopy category $\class{K}(R)$, where the morphism sets are chain homotopy classes of maps. The construction of this model structure was the subject of the paper~\cite{homotopy category of chain complexes is a homotopy category} and its description in terms of cotorsion pairs was pointed out in~\cite{hovey}. This model structure is a nice example of a Frobenius model structure.

Let $\text{Ch}(R)_{dw}$ be the exact category $(\cat{A},\class{E})$, where $\cat{A}$ is the category $\text{Ch}(R)$ of chain complexes of $R$-modules and $\class{E}$ is the class of all degreewise split exact sequences. Then one can check that $\text{Ch}(R)_{dw}$ is a weakly idempotent complete exact category. Let $\class{A}$ denote the class of all complexes and $\class{W}$ the class of all split exact complexes (i.e., contractible complexes).

\begin{corollary}\label{cor-the cotorsion pairs inducing classical homotopy theory}
The class $\class{W}$ of contractible complexes is a thick subcategory of $\text{Ch}(R)_{dw}$. Both $(\class{A},\class{W})$ and $(\class{W},\class{A})$ are complete cotorsion pairs with respect to $\text{Ch}(R)_{dw}$. The corresponding model structure on $\text{Ch}(R)_{dw}$ is described as follows. The cofibrations (resp. trivial cofibrations) are the degreewise split monomorphisms (resp. split monomorphisms with contractible cokernel) and the fibrations (resp. trivial fibrations) are the degreewise split epimorphisms (resp. split epimorphisms with contractible kernel). The weak equivalences are the usual homotopy equivalences. We note the following properties of this model structure:
\begin{enumerate}
\item The model structure is Frobenius.
\item The homotopy relation coincides with the usual notion of chain homotopy and two maps are chain homotopic if and only if their difference factors through a contractible complex.
\item $\text{HoCh}(R)_{dw} = \class{K}(R)$.
\item $\text{HoCh}(R)_{dw}(X,\Sigma^nY) \cong \Ext^n_{dw}(X,Y) \cong \text{HoCh}(R)_{dw}(\Sigma^{-n}X,Y)$, where here $\Ext^n_{dw}(X,Y)$ denotes the usual subfunctor of $\Ext^n(X,Y)$ consisting of the dimension-wise split exact sequences.
\end{enumerate}
\end{corollary}

\begin{proof}
We leave it to the reader to show that the class $\class{W}$ of contractible complexes is closed under direct summands and satisfies the two out of three property in $\text{Ch}(R)_{dw}$ making $\class{W}$ a thick subcategory of $\text{Ch}(R)_{dw}$.

Recall that by definition, an object $I$ in an exact category $(\class{A},\class{E})$ is injective if any admissible monomorphism $I \rightarrowtail Z$ splits (has a left inverse). Dual for projectives. We argue that $\class{W}$ are the pro-injective objects in $\text{Ch}(R)_{dw}$. To see this, we will show $\class{W}$ are the injective objects. A dual argument will show they are also the projective objects. Let $W \in \class{W}$ and consider an admissible monomorphism $W \rightarrowtail Z$. Then this is a degreewise split monomorphism and one can see that $Z$ is actually isomorphic to the mapping cone $C(f)$ for some chain map $f \mathcolon \Sigma^{-1}Y \ar W$. But since $W$ is contractible, $f$ is null homotopic and this implies that $W \rightarrowtail Z$ is a split monomorphism. On the other hand, suppose $I$ is injective, so that any admissible monomorphism $I \rightarrowtail Z$ in $\text{Ch}(R)_{dw}$ is a split mono. Then in particular, the short exact sequence $0 \ar I \ar C(1_I) \ar \Sigma I \ar 0$ is an admissible mono and must split. Thus $I$ is a direct summand of $C(1_I)$. But since the mapping cone of an identity map is always contractible (exercise~1.5.2 of~\cite{weibel}) we see that $I$ is a direct summand of a contractible complex. Therefore $I$ is also contractible.

Note that the mapping cone construction $C(1_W)$ also shows that there are enough pro-injective objects in $\text{Ch}(R)_{dw}$. It follows that both $(\class{A},\class{W})$ and $(\class{W},\class{A})$ are complete cotorsion pairs in $\text{Ch}(R)_{dw}$. By Corollary~\ref{cor-Hovey's theorem for exact categories} and Proposition~\ref{cor-characterizations of homotopic maps in inj. proj. Frob. model structures} they induce a Frobenius model structure with (trivial) cofibrations and (trivial) fibrations as we described. Note that Proposition~\ref{cor-characterizations of homotopic maps in inj. proj. Frob. model structures} says that two chain maps $f,g \mathcolon X \ar Y$ are homotopic in this model structure if and only if $g-f$ factors through a contractible complex. This corresponds to the usual fact that a map $X \ar Y$ is null homotopic if and only if it extends to a map from $C(1_X)$ to $Y$ as in exercise~1.5.2 of~\cite{weibel}. So the homotopy relation in this model category coincides with the usual notion of chain homotopy equivalence.

Now $\Ext_{dw}$ is the usual derived functor of $\Hom_{\text{Ch}(R)_{dw}}$ and so $\Ext_{dw}$ can be computed as maps in $\text{HoCh}(R)_{dw}$. In particular, we have the isos $\Ext^n_{dw}(X,Y) \cong \text{HoCh}(R)_{dw}(QX,R\Sigma^nY) \cong \text{HoCh}(R)_{dw}(X,\Sigma^nY)\cong \text{Ch}(R)(X,\Sigma^nY)/\sim$.
\end{proof}


\end{document}